\documentclass{article}
\usepackage[utf8]{inputenc}

\usepackage{fullpage}
\usepackage{amsfonts}
\usepackage{amsmath}
\usepackage{amssymb}
\usepackage{amsthm}

\usepackage{bm}
\usepackage{graphics}
\usepackage{graphicx}

\usepackage{hyperref}

\usepackage{subcaption}

\usepackage{latexsym,amsfonts,amssymb,amsmath,amsthm}

\usepackage{booktabs} 
\usepackage{graphicx}
\usepackage{pgfplots}
\usepackage[all]{nowidow}
\usepackage{soul}
\usepackage[utf8]{inputenc}
\usepackage{tikz}
\usetikzlibrary{automata}
\usepackage{multicol,comment}
\usepackage{algpseudocode,algorithm,algorithmicx}

\usepackage{cite}
\usepackage{hyperref}
\usepackage[inline]{enumitem} 

\usepackage{ulem}

\definecolor{blue}{HTML}{1F77B4}
\definecolor{orange}{HTML}{FF7F0E}
\definecolor{green}{HTML}{2CA02C}

\pgfplotsset{compat=1.14}

\newtheorem{tm}{Theorem}[section]

\newtheorem{lem}[tm]{Lemma}

\newtheorem{rk}[tm]{Remark}

\numberwithin{equation}{section}
\numberwithin{tm}{section}

\title{Global dynamics of epidemic network models via construction of Lyapunov functions 
}

 \author{Rachidi B. Salako\footnote{rachidi.salako@unlv.edu; Department of Mathematical Sciences, University of Nevada Las Vegas, Las Vegas, USA} \quad and \quad Yixiang Wu\footnote{yixiang.wu@mtsu.edu; Department of Mathematical Sciences, Middle Tennessee State University,
  Murfreesboro, Tennessee 37132, USA} }
\date{}

\begin{document}

\maketitle

\begin{abstract} 
In this paper, we study the global dynamics of epidemic network models with standard incidence or mass-action transmission mechanism, when the dispersal of either the susceptible or the infected people is controlled. The connectivity matrix of the model is not assumed to be symmetric. Our main technique to study the global dynamics is to construct novel Lyapunov type functions.
\end{abstract}

\section{Introduction}

In this paper, we study the global dynamics of the following Susceptible-Infected-Susceptible (SIS)   epidemic network (patch) model
 \begin{equation}\label{model-eq1}
 \left\{
 \begin{array}{lll}
\displaystyle S_i'&=d_S\displaystyle\sum_{j\in\Omega, j\neq i}(L_{ij} S_j-L_{ji} S_i)-f_i(S_i, I_i)+ \gamma_{i} I_{i},  & i\in\Omega,\ t>0, \cr 
\displaystyle I_i'&=d_I\displaystyle\sum_{j\in\Omega, j\neq i}(L_{ij} I_{j}-L_{ji} I_{i})+f_i(S_i, I_i)-\gamma_{i} I_{i}, & i\in\Omega,\ t>0,    
\end{array}
\right.
\end{equation}
where $\bm S=(S_1, \dots, S_n)$ and $\bm I=(I_1, \dots, I_n)$ denote the number of susceptible and infected people living in $n-$patches (e.g., cities, countries, etc.) respectively;  $n\ge 2$ is the number of patches; $\Omega=\{1,2,\cdots,n\}$ is the collection of patch labels; $\gamma_i$ is the disease recovery rate; and $f_i(S_i, I_i)$ describes the interaction of susceptible and infected people.
The coefficient $L_{ij}$, $j\neq i$, is the degree of movement of individuals from patch $j$ to $i$, and $d_S$ and $d_I$ are the dispersal rates of susceptible and infected people, respectively. For simplicity,  let $L_{ii}=-\sum_{j\in\Omega, j\neq i} L_{ji}$, which is the total degree of movement out from patch $i$. The matrix $\mathcal{L}=(L_{ij})$ is called the connectivity matrix, which will \textit{not} be assumed to be symmetric. We can rewrite model \eqref{model-eq1} as
 \begin{equation}\label{model-eq2}
 \left\{
 \begin{array}{lll}
\displaystyle S_i'&=d_S\displaystyle\sum_{j\in\Omega}L_{ij} S_j-f_i(S_i, I_i)+ \gamma_{i} I_{i},  & i\in\Omega,\ t>0, \cr 
\displaystyle I_i'&=d_I\displaystyle\sum_{j\in\Omega}L_{ij} I_{j}+f_i(S_i, I_i)-\gamma_{i} I_{i}, & i\in\Omega,\ t>0.    
\end{array}
\right.
\end{equation}
Adding up all the equations in \eqref{model-eq1}, we obtain $\frac{d}{dt}\sum_{i\in\Omega}(S_i+I_i)=0$. This means that the total population $N:=\sum_{i\in\Omega}(S_i+I_i)$ remains a constant for all $t\ge 0$.

Model \eqref{model-eq1} when $d_S>0$ and $d_I>0$ has been studied recently in order to understand the joint impact of environmental heterogeneity and population movement on disease transmission. In \cite{allen2007asymptotic}, Allen \textit{et al.} studied \eqref{model-eq1} with standard incidence transmission mechanism $f_i=\beta_i S_iI_i/(S_i+I_i)$ ($\beta_i$ is called the disease transmission rate) and symmetric connectivity matrix $\mathcal{L}$. In particular, the authors defined a basic  reproduction number $\mathcal{R}_0$ and showed that $\mathcal{R}_0$ is a threshold parameter: the solution converges to the disease free equilibrium if $\mathcal{R}_0<1$ and the model has a unique endemic equilibrium (EE) (i.e. positive equilibrium) if $\mathcal{R}_0>1$. Moreover, they showed that the disease component of the EE converges to zero if $d_S\to 0$ when $\beta_i-\gamma_i$ changes sign for $i\in\Omega$. Biologically, this result indicates that one may control the disease by limiting the movement of susceptible people. In contrast, Li and Peng \cite{li2019dynamics} show that one may not eliminate the disease by limiting the movement rate of infected people $d_I$. In \cite{chen2020asymptotic}, Chen \textit{et al.} show that the above results remain true if the connectivity matrix $\mathcal{L}=(L_{ij})$ is not symmetric.  
{Model \eqref{model-eq1} with the mass-action incidence transmission mechanism $f_i=\beta_i S_iI_i$ has also received some attention recently (see  \cite{li2023sis} for symmetric $\mathcal{L}$  and \cite{Salako2023dynamics} for asymmetric $\mathcal{L}$).} 
These results show that whether the disease can be eliminated by limiting $d_S$ is depending on the total population size $N$ and the infected people may concentrate on the patches of the highest risk if $d_I\to 0$. Moreover, it has been shown in \cite{Salako2023dynamics} that the model {with  mass-action incidence  transmission mechanism} may have multiple EE solutions when $\mathcal{R}_0<1$. Therefore, the predictions on the consequences of disease control by limiting population movement are depending on the chosen transmission mechanism, the population size, and the control strategy.    We remark that Gao \textit{et al.} \cite{gao2020does,gao2021impact} studied the impact of population movement on the number of infected cases using model \eqref{model-eq1} {with the standard incidence transmission mechanism}, and 
the related reaction-diffusion models have attracted many research interests
(see, e.g., \cite{Allen,Li2018,WuZou,castellano2022effect}).

The global dynamics of epidemic network models have attracted tremendous research interests recently. In particular, a graph theoretical method based on Kirkland's matrix tree theorem to construct   Lyapunov functions for network models has been proposed in \cite{guo2006global, li2010global}. The idea of this method is to take a weighted combination of the Lyapunov functions for the decoupled  local models (i.e., the system without dispersal), and the weights are determined by a Laplacian matrix that is related to the network structure. 
This method has been applied to obtain the global stability of positive equilibrium for many biological network models (see, e.g., \cite{shuai2013global,zhang2015graph,shu2012global,guo2012global}). 

The main objective of the current paper is to investigate the global dynamics of model \eqref{model-eq1} with $d_S=0$ or $d_I=0$. As pointed out in the recent work \cite{Salako2023degenerate} on the reaction-diffusion counterpart of model \eqref{model-eq1}, the asymptotic profiles of the EE can reflect the consequences of disease control by limiting population movement if the solutions of \eqref{model-eq1} always establish at some EE and the disease dynamics happens in a faster time scale than the disease control. However, if the disease control happens in a faster time scale, then the global dynamics of \eqref{model-eq1} with $d_S=0$ or $d_I=0$ may better tell the impact of the disease control strategies. As in  \cite{Salako2023degenerate}, we will construct Lyapunov type functions to study the global dynamics of the solutions. The  Lyapunov functions constructed by us do not use the graph-theoretical method, and we will conquer the difficulties arising from  asymmetric $\mathcal{L}$ by proving a useful result, i.e., Theorem \ref{theorem_positive}. 

The rest of the paper is organized as follows. In section 2, we  present some preliminary results. In section 3, we  study the global dynamics of \eqref{model-eq1} with the mass-action incidence mechanism. In section 4, we  study the global dynamics of the model with the standard incidence mechanism. 


\section{Preliminaries}\label{Sec2}

\subsection{Notations}
We will use a bold letter to represent   a column vector in $\mathbb{R}^n$, and the corresponding no-bold form with a subscript $j$ will be the $j$th-component of the vector. 
Let $\bm 0:=(0, \dots, 0)^T$ and $\bm 1:=(1, \dots, 1)^T$. For $\bm X\in\mathbb{R}^n$, define 
 $$
{\bm X}_{m}:=\min_{j\in\Omega} X_j,\quad  {\bm X}_{M}:=\max_{j\in\Omega} X_j,\quad 
 \|\bm X\|_1:=\sum_{j=1}^n|X_{i}|,\quad \text{and}\quad \|\bm X\|_{\infty}:=\max_{j\in\Omega}| X_j|.
$$
We adopt the following convention: $\bm X\ge  \bm 0$ if $X_i\ge  0$ for all $i\in\Omega$, $\bm X\gg  \bm 0$ if $X_i>  0$ for all $i\in\Omega$, and $\bm X>\bm 0$ if $\bm X\ge \bm 0$ and $\bm X\neq \bm 0$. 
 For  $\bm X, \bm Y\in \mathbb{R}^n$, the Hadamard product of them is $ \bm X\circ \bm Y :=(X_1Y_1,\cdots, X_n Y_n)^T$ 
 and we define  ${\bm X}/{\bm Y}=(X_1/Y_1,\cdots,$
 $ X_n/ Y_n)^T$ if $Y_i\neq 0$ for all $i\in\Omega$.

For $n\times n$ square matrix $A$, let $\sigma(A)$ be the spectrum of $A$, $s(A)$ be the spectral bound of $A$, i.e.,
\begin{equation*}
    s(A):=\max\{\mathfrak{R}e(\lambda)\ :\ \lambda\in\sigma(A)\},
\end{equation*}
where $\mathfrak{R}e(\lambda)$ is the real part of $\lambda\in \mathbb{C}$, 
and  $\rho(A)$ be the spectral radius of $A$, i.e.,
$$
\rho(A):=\max\{|\lambda|\ :\ \lambda\in\sigma(A)\}.
$$

 Throughout this work, we shall suppose that the following  assumptions hold.

\medskip

\noindent{\bf (A1)} The matrix $\mathcal{L}=(L_{ij})_{i,j=1}^n$ is quasi-positive (i.e., $L_{ij}\ge 0$ for any $j\neq i$)  and irreducible with $L_{ii}=-\sum_{j\in\Omega, j\neq i} L_{ji}$ for each $i\in\Omega$.

\noindent{\bf (A2)} The initial data $(\bm S^0, \bm I^0)$ satisfies $\bm S^0\ge\bm 0$, $\bm I^0>\bm 0$  and   $\sum_{j\in\Omega}( S_{j}^0+ I^0_{j})=N$, where $N$ is a fixed positive constant. 

\noindent{\bf (A3)} $d_S\ge 0, d_I\ge 0$, $\bm \beta\gg \bm 0$ and $\bm \gamma\gg \bm 0$. 

\medskip

By {\bf (A1)} and the Perron-Frobenius theorem, we have 
\begin{lem}\label{lemma_eig}
    Suppose that {\bf (A1)} holds. Then $s(\mathcal{L})=0$ is a simple eigenvalue of $\mathcal{L}$, and there is an   eigenvector $\bm\alpha$ associated with $s(\mathcal{L})$ satisfying 
\begin{equation}\label{alpha-eq}
\mathcal{L}\bm\alpha=\bm 0,\quad \sum_{j\in\Omega}{\alpha}_j=1,\quad \text{and} \quad {\alpha}_j>0,\ \forall \ j\in\Omega,
\end{equation}  
and $\bm\alpha$ is the unique nonnegative eigenvalue of $\mathcal{L}$ satisfying $\sum_{j\in\Omega}{\alpha}_j=1$. 
\end{lem}

Throughout the paper, $\bm\alpha$ is fixed and given as in  Lemma \ref{lemma_eig}. By {\bf (A1)}-{\bf (A3)}, for each initial data $(\bm S^0, \bm I^0)$, model \eqref{model-eq1} has a unique nonnegative global solution, which  leaves the following compact set invariant 
$$
\mathcal{E}:=\Big\{(\bm S,\bm I)\in [\mathbb{R}^n_+]^2\ : \ \sum_{j\in\Omega}( S_j+ I_{j})=N\Big\}.
$$  
Moreover, if $d_I=0$, then $I_{j_0}(0)=0$ for some $j_0\in\Omega$ implies that $ I_{j_0}(t)=0$ for all $t>0$. However if $d_I>0$, by $\bm I^0>\bm 0$, $ I_{j}(t)>0$ for all $t>0$ and $j\in\Omega$. 

For convenience, define 
\begin{equation*}\label{local-risk}
    { r}_{j}
    :=\frac{ \gamma_{j}}{\beta_{j}}, \quad j\in\Omega,
\end{equation*}
which describes the local risk of patch $j$ when the network is decoupled.

\subsection{Some useful results}



    

We present an important result, Theorem \ref{theorem_positive}, which will help us to construct delicate Lyapunov functions later on.  An $n\times n$ real matrix $A=(a_{ij})$ is called \textit{line-sum symmetric} if $\sum_{j}a_{ij}=\sum_{j}a_{ji}$, i.e. the sum of the $i$-th row equals the sum of the $i$-th column for all $i=1,\dots, n$. 

\begin{lem}
\label{TC} Suppose that  {\bf (A1)} holds and $\mathcal{L}$ is line-sum symmetric.  Then $\sum_{i,j\in\Omega}L_{ij} X_i X_j\le 0$ for any $ \bm X\in \mathbb{R}^n$, where the equality holds if and only if $\bm X$ is a multiple of $\bm 1$.
\end{lem}
\begin{proof}
    Since $\mathcal{L}$ is line-sum symmetric, $-L_{ii}=\sum_{j\neq i, j\in\Omega} L_{ij}=\sum_{j\neq i, j\in\Omega} L_{ji}$ for any $i=1, \dots, n$. For any $\bm X\in\mathbb{R}^n$, we have 
    \begin{align*}
 \sum_{i, j\in\Omega}L_{ij} X_i X_j&=   \sum_{i, j\in\Omega, j\neq i}L_{ij} X_i X_j+\sum_{i\in\Omega} L_{ii} X_i^2\\
 &=- \frac{1}{2}\sum_{i, j\in\Omega, j\neq i}L_{ij}( X_i- X_j)^2+\frac{1}{2}\sum_{i, j\in\Omega, j\neq i}L_{ij} X_i^2+\frac{1}{2}\sum_{i, j\in\Omega, j\neq i}L_{ij} X_{j}^2+\sum_{i\in\Omega} L_{ii} X_i^2\\
 &=-\frac{1}{2}\sum_{i, j\in\Omega, j\neq i}L_{ij}( X_i- X_j)^2+\frac{1}{2}\sum_{i, j\in\Omega, j\neq i}L_{ij} X_i^2+\frac{1}{2}\sum_{i, j\in\Omega, j\neq i}L_{ji} X_{i}^2+\sum_{i\in\Omega} L_{ii} X_i^2\\
  &=-\frac{1}{2}\sum_{i, j\in\Omega, j\neq i}L_{ij}( X_i- X_j)^2-\frac{1}{2}\sum_{i\in\Omega}L_{ii} X_i^2-\frac{1}{2}\sum_{i\in\Omega}L_{ii} X_{i}^2+\sum_{i\in\Omega} L_{ii} X_i^2\\
   &=-\frac{1}{2} \sum_{i, j\in\Omega, j\neq i} L_{ij}( X_i- X_j)^2\le 0.
    \end{align*}
 Since $L$ is irreducible, the equality holds if and only if $ X_1=\dots= X_n$, i.e., $\bm X$ is a multiple of $\bm 1$.    
\end{proof}

\begin{rk}
    By Lemma \ref{TC}, if $\mathcal{L}$ satisfies  {\bf (A1)} and is line-sum symmetric, then $(\mathcal{L}+\mathcal{L}^T)/2$ is negative semidefinite. 
\end{rk}

\begin{tm}\label{theorem_positive}
 Suppose that  {\bf (A1)} holds. Let $\bm\alpha$  be defined as above and $\bm \theta:=1/\bm\alpha$. Then $\sum_{i, j}\theta_iL_{ij} X_i X_j\le 0$ for any $\bm X\in\mathbb{R}^n$, where the equality holds if and only if $\bm X$ is a multiple of $\bm \alpha$. 
\end{tm}
\begin{proof}
   Let $D=\text{diag}(\bm\alpha)$. Then, we have $\mathcal{L}D=(L_{ij}\alpha_j)_{i,j=1}^n$. The sum of the $j$-th column of $\mathcal{L}D$ is $\sum_{i}L_{ij}\alpha_j=\alpha_j\sum_{i}L_{ij}=0$ for any $1\le j\le n$, since the sum of each column of $\mathcal{L}$ is zero.  The sum of the $i$-th row of $\mathcal{L}D$ is $\sum_{j}L_{ij}\alpha_j=0$ for any $1\le i\le n$, since $\alpha$ is an eigenvector of $\mathcal{L}$ corresponding with eigenvalue zero. Therefore, $\mathcal{L}D$ is line-sum symmetric, where the sum of each column and row is zero. By Lemma \ref{TC}, we have $\sum_{i,j}L_{ij}\alpha_j  Y_i Y_j\le 0$ for any $\bm Y\in \mathbb{R}^n$, where the equality holds if and only if $\bm Y$ is a multiple of $\bm 1$.

   We notice that $\sum_{i,j}L_{ij}\alpha_j Y_i Y_j=\sum_{i,j}\theta_iL_{ij}(\alpha_i Y_i)(\alpha_j Y_j)$. Since $\bm\alpha\gg\bm 0$ and $\bm Y\in\mathbb{R}^n$ were arbitrary, we have  $\sum_{i, j}\theta_iL_{ij} X_i X_j\le 0$ for any $\bm X\in\mathbb{R}^n$, where the equality holds if and only if $\bm X$ is a multiple of $\bm \alpha$. 
\end{proof}

\begin{rk}
     Suppose that $L$ satisfies  {\bf (A1)}. Let $\bm\alpha$ be a positive eigenvector of $\mathcal{L}$ and $\bm\theta=1/\bm\alpha$.  Let $D_1=\text{diag}(\bm\alpha)$ and $D_2=\text{diag}(\bm\theta)$.  By Lemma \ref{TC} and Theorem \ref{theorem_positive}, the two matrices $(\mathcal{L}D_1+ (\mathcal{L}D_1)^T)/2$ and $(D_2\mathcal{L}+ (D_2\mathcal{L})^T)/2$ are negative semidefinite. 
\end{rk}

 Next, we collect some useful results. 
\begin{lem}[\cite{Salako2023dynamics}]\label{lem0}
    Suppose that  {\bf (A1)} holds. Let $d$ be a positive number and $\bm F :\mathbb{R}_+\to \mathbb{R}^n$ be a continuous map satisfying $\|\bm F(t)\|_1\to 0$ as $t\to\infty$. Let  $\bm X(t)$ be a bounded solution of the system
    \begin{equation*}
       \begin{cases} \bm X'(t)=d\mathcal{L}\bm X(t) +\bm F(t),\ t>0,\cr 
        \bm X(0)=\bm X^0\in\mathbb{R}^n.
        \end{cases}
    \end{equation*}
    Then $  
    \lim_{t\to\infty}\|\bm X(t)-(\sum_{j\in\Omega} X_j(t))\bm\alpha\|_1=0.$  In particular, if $\bm F(t)={\bf 0} $ for all $t\ge 0$,  then $\bm X(t)\to (\sum_{j\in\Omega}X^0_j)\bm\alpha$ as $t\to\infty$.
\end{lem}

\medskip

The following is a Harnack's type inequality.

\begin{lem}\cite[Lemma 3.1]{DBS2023}\label{Harnck-lemma} Suppose that  {\bf (A1)} holds. Let $d>0$ and $ \bm M\in C(\mathbb{R}_+:[\mathbb{R}]^n)$ such that 
\begin{equation*}
    m_{\infty}:=\sup_{t\ge 0}\|\bm M(t)\|_{\infty}<\infty.
\end{equation*}
Then there is a positive number $c_{d,m_{\infty}}$ such that any nonnegative solution $\bm U(t)$ of 
\begin{equation}\label{Harnack-eq1}
    \bm U'=d\mathcal{L}\bm U +\bm M(t)\circ \bm U, \quad t>0,
\end{equation}
satisfies
\begin{equation}\label{Harnack-eq2}
    \|\bm U(t)\|_{\infty}\le c_{d,m_{\infty}}\bm U_{m}(t),\quad \forall\ t\ge 1.
\end{equation}   
\end{lem}

    The following result is well-known.
\begin{lem}\label{lemma_inf0}
    Suppose that $\phi:\mathbb{R}_+\to\mathbb{R}_+$ is  H\"older continuous and $\int_0^\infty\phi(t)dt<\infty$. Then $\lim_{t\to\infty}\phi(t)=0$.
\end{lem}
\medskip

 \section{Model with the mass-action incidence mechanism  }
 In this section, we focus on the model with the   mass-action incidence mechanism, that is 
 \begin{equation}\label{model-mass-action}
     \begin{cases}
         \bm S'=d_S\mathcal{L}\bm S-\bm\beta\circ\bm S\circ\bm I+\bm\gamma\circ \bm I, \quad & t>0,\cr 
         \bm I'=d_I\mathcal{L}\bm I+\bm\beta\circ\bm S\circ \bm I-\bm\gamma\circ \bm I, \quad & t>0.
     \end{cases}
 \end{equation}
  First, we will discuss the scenario of a complete limitation of the movement of infected population, i.e $d_I=0$, and obtain Theorem \ref{thm-mass-di}. Then, we will present our result on the scenario of a complete limitation of the movement of susceptible population, i.e $d_S=0$, and  establish Theorem \ref{thm-mass-ds}.

 We introduce the low-risk, moderate-risk, and high-risk $\tilde{H}^-$, $\tilde{H}^0$, and $\tilde{H}^+$ patches as follows:
$$
\tilde{H}^-=\{i\in\Omega: 1<r_i/(N\alpha_i)\},\quad  \tilde{H}^0:=\{i\in\Omega  : 1=r_i/(N\alpha_i) \},\quad \text{and}\quad  \tilde{H}^{+}:=\{i\in\Omega : 1>r_i/(N\alpha_i) \}.
$$
The sets $\tilde{H}^{-}$, $\tilde{H}^{0}$, and $\tilde{H}^{+}$ partition $\Omega$ into disjoint subsets, and  either of these sets can be empty. Note that if the system is decoupled, $R_i=N\alpha_i/r_i=N\alpha_i\beta_i/\gamma_i$ is the basic reproduction number of patch $i$.

  Given an initial data $(\bm S^0,\bm I^0)$, depending on the initial distribution of the infected population, we introduce the set
 $$\Omega^0_{\bm I^0}:=\{i\in\Omega: \ I^0_{i}=0\}\quad  \text{and}\quad \Omega^+_{\bm I^0}:=\{i\in\Omega: \  I_{i}^0>0\}$$
 and the quantity
 $$
\tilde r_{m}=\min_{i\in\Omega_{\bm I^0}^+} \frac{r_i}{N\alpha_i}.
 $$
 The patches realizing the above minimum will be called the patches of the highest-risk. 
 
\begin{tm}\label{thm-mass-di} Suppose that {\bf (A1)-(A3)} holds, $d_S>0$ and $d_I=0$. Then every solution $(\bm S(t),\bm I(t))$ of \eqref{model-mass-action} satisfies 
\begin{equation}\label{IH}
\lim_{t\to\infty}\|\bm S(t)-(\sum_{i\in\Omega}S_i(t)){\bm \alpha}\|_{\infty}= 0\quad\text{and}\quad\lim_{t\to\infty}\sum_{j\in \tilde{H}^{-}{\cup \tilde{H}^0}\cup\Omega^0_{\bm I^0}} I_j(t)=0.
\end{equation} 
Furthermore, the following conclusions hold.
\begin{itemize}
    \item[\rm (i)]  If $\tilde{H}^+\cap \Omega^+_{\bm I^0}=\emptyset$, then  $\|\bm S(t)-{N{\bm \alpha}}\|_\infty\to 0$ and $\|\bm I(t)\|_\infty\to 0$ as $t\to\infty$;

\item[\rm (ii)]
If $\tilde{H}^+\cap \Omega^+_{\bm I^0}\ne\emptyset$, then $\|\bm S(t)-N\tilde{r}_m{\bm\alpha}\|_{\infty}\to 0$, $\sum_{j\in\Omega} I_j(t)dx\to N-|\Omega|\tilde{r}_m$, and $ I_j(t)\to 0 $ on  $\{ j\in\Omega :\ r_j/(N\alpha_j)\ne \tilde{r}_m\}$ as $t\to\infty$. 
\end{itemize}
    
\end{tm}
\begin{proof}
    Define 
    $$
   V(\bm S, \bm I) =\sum_{i\in\Omega} \theta_i\left(\frac{ S_i^2}{2}+r_iI_i\right),
    $$
    where $ \theta_i=1/\alpha_i$ for $i\in\Omega$. It is easy to check that 
    $$
    \dot{V}(\bm S, \bm I)=d_S\sum_{i,j\in\Omega} \theta_iL_{ij} S_i S_j-\sum_{i\in\Omega}\theta_i\beta_i\left( S_i-r_i\right)^2 I_i.
    $$
By Theorem \ref{theorem_positive}, we have $\sum_{i,j\in\Omega} \theta_iL_{ij} S_i S_j\le 0$ for all $\bm S\in\mathbb{R}^k$ and the equality holds if and only if $\bm S$ is a multiple of $\bm \alpha$. 

Denote $\mathcal{M}:=\{(\bm S, \bm I)\in \mathcal{E}:\ \dot{V}(\bm S, \bm I)=0\}$. Then, 
$$
\mathcal{M}=\left\{(\bm S, \bm I)\in \mathcal{E}:\ S_i=r_i \ \text{or} \ I_i=0 \ \ \forall i\in\Omega\ \ \text{and}\ \bm S=a\bm\alpha \text{ for some } a\in\mathbb{R} \right\}.
$$
The maximal connected invariant subsets of $\mathcal{M}$ are
$$
\mathcal{M}_i:=\left\{(\bm S, \bm I)\in \mathcal{E}:\  \bm S=\frac{r_i}{\alpha_i}\bm\alpha  \ \ \text{and}\ \  I_j=0\ \ \text{if}\ \frac{r_j}{\alpha_j} \neq \frac{r_i}{\alpha_i}\ \forall j\in\Omega \right\}, \ \ i\in\Omega,
$$
and 
$$
\mathcal{M}_0:=\left\{(\bm S, \bm I)\in \mathcal{E}:\  \bm S=\frac{N}{|\Omega|}\bm\alpha  \ \ \text{and}\ \  \bm I=\bm 0 \right\}.
$$
Hence by the invariance principle \cite[Theorem 4.3.4]{DanHenry}, $(\bm S(t), \bm I(t))\to \mathcal{M}_i$ for some $i=0, 1, ..., n$ as $t\to\infty$. 
It follows from the equation of $\bm S$ in \eqref{model-mass-action} and Lemma \ref{lem0} that 
\begin{equation}\label{Sinf}
\lim_{t\to\infty}\Big\|\bm S(t)-{\Big(\sum_{j\in\Omega}S_j(t)\Big)}{\bm \alpha}\Big\|_{\infty}= 0.
\end{equation}
If $\tilde{H}^-\neq\emptyset$, by its definition, we can choose  $\varepsilon>0$ such that 
$$
\eta_\varepsilon:=\max_{i\in\tilde{ H}^-}\left\{\beta_i\left((N+\varepsilon)\alpha_i-r_i\right)\right\}<0. 
$$
By \eqref{Sinf} and $\sum_{i\in\Omega}(S_i+I_i)=N$ for all $t\ge 0$, there is $t_\varepsilon>0$ such that 
\begin{equation}\label{salpha}
    \bm S(t)\le \Big(N-\sum_{i\in\Omega}I_i+\varepsilon\Big) \bm\alpha, \ \ \forall \; t\ge t_\varepsilon. 
\end{equation}
It follows that 
$$
I_i'=\beta_i\left(S_i-r_i\right)I_i\le \beta_i\left((N+\epsilon)\alpha_i-r_i\right)I_i\le \eta_\varepsilon I_i, \ \ \forall \ t\ge t_\varepsilon, \ i\in \tilde{H}^-.
$$
Noticing $\eta_\varepsilon <0$, we have 
$$
I_i(t)\le I_i(t_\varepsilon) e^{\eta_\varepsilon t}\to 0\ \ \text{as} \ t\to\infty, \quad\forall\; i\in\tilde H^-. 
$$
Next, if $\tilde{H}^0\ne\emptyset$, by  \eqref{salpha},
$$
I_i'=\beta_i(S_i-r_i)I_i\le \beta_i((N-I_i+\varepsilon)\alpha_i-r_i)I_i=\beta_j\alpha_i(\varepsilon-I_i)I_i,\quad r\ge t_{\varepsilon},\ i\in \tilde{H}^0.
$$
Hence
$$
\limsup_{t\to\infty}I_i(t)\le \varepsilon\quad \forall\ i\in \tilde{H}^0.
$$
Letting $\varepsilon\to 0$ in the last inequality, we obtain that $I_i(t)\to 0$ as $t\to\infty$ for each $i\in \tilde{H}^0.$
Since $I_i(t)=0$ for all $t\ge 0$ for all $i\in\Omega^0_{\bm I^0}$, \eqref{IH} holds.

It remains to prove (i)-(ii). We first show that 
\begin{equation}\label{sup}
   \limsup_{t\to\infty} \sum_{i\in\Omega} I_i(t)\le N(1-\tilde r_m)_+. 
\end{equation}
Since $I_i(t)=0$ for all $t\ge 0$ for all $i\in\Omega^0_{\bm I^0}$, it suffices to prove \eqref{sup} with $\Omega$ replaced by $\Omega^+_{\bm I^0}$. To see this, by \eqref{salpha},
\begin{align*}
 I_i'=\beta_i\left(S_i-r_i\right)I_i&\le \beta_i\Big(\Big(N-\sum_{i\in\Omega^+_{\bm I^0}}I_i+\epsilon\Big)\alpha_i-r_i\Big)I_i\\
 &\le \alpha_i\beta_i\Big(N(1-\tilde{r}_m)_++\varepsilon-\sum_{i\in\Omega^+_{\bm I^0}}I_i \Big)I_i,\ \ \forall \ t\ge t_\varepsilon, \ i\in \Omega^+_{\bm I^0}.   
\end{align*}
Define 
$$
F(t)=\frac{\sum_{i\in\Omega^+_{\bm I^0}}\alpha_i\beta_iI_i(t)}{\sum_{i\in\Omega^+_{\bm I^0}}I_i(t)}, \ \ t\ge 0.
$$
 It then follows that 
$$
\frac{d}{dt} \sum_{i\in\Omega^+_{\bm I^0}} I_i \le F(t)\Big(N(1-\tilde{r}_m)_++\varepsilon-\sum_{i\in\Omega^+_{\bm I^0}}I_i \Big)\sum_{i\in\Omega^+_{\bm I^0}}I_i,\ \ \forall \ t\ge t_\varepsilon
$$
Noticing $0<\min\{\alpha_i\beta_i:\ x\in\Omega^+_{\bm I^0}\}\le F(t)\le \max\{\alpha_i\beta_i:\ x\in\Omega^+_{\bm I^0}\}$, we have 
$$
  \limsup_{t\to\infty} \sum_{i\in\Omega^+_{\bm I^0}} I_i(t)\le N(1-\tilde{r}_m)_++\varepsilon. 
$$
Since $\varepsilon>0$ was arbitrary, \eqref{sup} holds.

(i) If $H^+\cap \Omega^+_{\bm I^0}=\emptyset$, then $(1-\tilde{r}_m)_+=0$. By \eqref{sup}, we have $\lim_{t\to\infty}\|\bm I(t)\|_\infty=0$. Since $\sum_{i\in\Omega}(S_i+I_i)=N$ for all $t\ge 0$, $\sum_{i\in\Omega}S_i(t)\to N$ as $t\to\infty$. Then it follows from \eqref{Sinf} that $\|\bm S(t)-{N{\bm \alpha}}\|_\infty\to 0$ as $t\to\infty$. 

(ii) If $H^+\cap \Omega^+_{\bm I^0}\ne\emptyset$, then $1>\tilde{r}_m$.
By  \eqref{sup} and $\sum_{i\in\Omega} (S_i(t)+I_i(t))=N$ for all $t\ge 0$, 
\begin{equation*}
   \liminf_{t\to\infty} \sum_{i\in\Omega} S_i(t)\ge N{\tilde{r}_m}. 
\end{equation*}
By \eqref{Sinf}, we have 
\begin{equation*}
   \liminf_{t\to\infty}  S_i(t)\ge N{\tilde{r}_m}\alpha_i, \ \ \forall \ i\in\Omega. 
\end{equation*}

Now we claim that there exists $i_0\in\Omega$ such that 
\begin{equation}\label{infSi}
   \liminf_{t\to\infty}  S_{i_0}(t)= N{\tilde{r}_m}\alpha_{i_0}. 
\end{equation}
Assume to the contrary that the claim is false. Then there exist $\varepsilon', t_{\varepsilon'}>0$ such that 
$$
S_i(t)\ge N{\tilde{r}_m}\alpha_i+\varepsilon', \ \ \forall \ t\ge t_{\varepsilon'}, \ i\in\Omega.
$$
Let $i_0\in \Omega^+_{\bm I^0}$ be such that ${\tilde{r}_m}=r_{i_0}/(N\alpha_{i_0})$. Then,
$$
I'_{i_0}=\beta_{i_0}(S_{i_0}-r_{i_0})I_{i_0}\ge \varepsilon'\beta_{i_0}I_{i_0}, \ \ \forall\;t\ge t_{\varepsilon'}.
$$
This implies $I_{i_0}(t)\ge I_{i_0}(t_{\varepsilon'})e^{\varepsilon'\beta_{i_0}(t-t_{\varepsilon'})}$, $t\ge t_{\varepsilon'}$, which contradicts the boundedness of the solution. This proves \eqref{infSi}. 

By \eqref{infSi},  $(\bm S(t), \bm I(t))\to \mathcal{M}_{i_0}$ as $t\to\infty$. This proves (ii). 
\end{proof}
\begin{rk}

By Theorem \ref{thm-mass-di},  if $\tilde{H}^+\cap \Omega^+_{\bm I^0}=\emptyset$, the disease can be eliminated by limiting the movement of infected people when the model has mass-action mechanism. However   if $\tilde{H}^+\cap \Omega^+_{\bm I^0}\neq\emptyset$, 
the infected people will concentrate on the patches of the highest risk, $\{ j\in\Omega :\ r_j/(N\alpha_j)\ne \tilde{r}_m\}$,  when limiting the movement of them. These predictions largely agree with those in \cite{li2023sis,Salako2023dynamics} for the model with mass-action mechanism,  which are based on the asymptotic profiles of the EE as $d_I\to 0$. The main difference is that our results are depending on the distribution of initial data $\bm I^0$. 
\end{rk}


Next, we examine the scenario of a total restriction of the movement of the susceptible population. In this direction, our result reads as follows. 

\begin{tm}\label{thm-mass-ds}Suppose that {\bf (A1)-(A3)} holds, $d_S=0$ and $d_I>0$. Then any solution $(\bm S(t), 
\bm I(t))$  of \eqref{model-mass-action} satisfies $\|\bm S(t)-\bm S^*\|_{\infty}\to 0$ as $t\to\infty$ for some $\bm S^*\gg \bm 0$. Furthermore, exactly one of the following statements holds:
\begin{itemize}
    \item[\rm (i)] $\bm S^*=\bm\lambda^*\circ \bm S^0+({\bm 1}-\bm \lambda^*)\circ \bm r$ for some $\bm \lambda^*\in\mathbb{R}^n$   with ${\bm 0}\ll \bm\lambda^*\ll{\bm 1}$, $\sum_{j\in\Omega}S^*_j=N$ and $s(d_I\mathcal{L}+ {\rm diag}(\bm\beta\circ \bm\lambda^*\circ (\bm S^0-\bm r)))\le 0 $, and $\bm I(t)\to\bm 0$ as $t\to\infty$.

    \item[\rm(ii)] $\bm S^*=\bm r$ and  $\|\bm I(t)-\bm I^*\|_\infty\to0$ as $t\to\infty$, where $\bm I^*=(N-\sum_{j\in\Omega} r_j)\bm\alpha\gg \bm 0$.
\end{itemize}
Moreover, {\rm (i)} holds if $N\le \sum_{j\in\Omega}r_{j}$ while {\rm (ii)} holds if $N>N^*_{\bm S^0,\bm r}$, where $N^*_{\bm S^0,\bm r}$ is given by 
\begin{equation}
    N^*_{\bm S^0,\bm r}:=\sup\left\{\bm\lambda^*\cdot \bm S^0+(1-\bm\lambda^*)\cdot \bm r:\  {\bm 0}\le \lambda^*\le {\bf 1}\quad \text{and}\quad s(d_I\mathcal{L}+{\rm diag}(\bm \beta\circ\bm\lambda^*\circ(\bm S_0-\bm r)))\le 0\right\}.
\end{equation}
\end{tm}
\begin{proof}
 Let  $   J_i(t):=\int_0^t  I_i(\tau)d\tau $ for  $i\in\Omega$ and  $t\ge 0$ and define 
$$
\tilde{ J}_i:=\int_0^\infty  I_i(\tau)d\tau=\lim_{t\to\infty} J_i(t)\in (0, \infty]. 
$$ 
We consider two cases.\\ 
\textbf{Case 1.} $\tilde{ J}_{i_0}<\infty$ for some $i_0\in\Omega$. 
By Lemma \ref{Harnck-lemma}, there exists $C_1>0$ such that 
\begin{equation}\label{maxmi}
    \bm I_M(t)\le C_1  I_i(t), \quad \forall\;t\ge 1, i\in\Omega. 
\end{equation}
Therefore, 
$$
\tilde{ J}_i=\tilde{ J}_i(1)+\int_1^\infty I_i(\tau)d\tau \le \tilde{ J}_i(1)+C_1\int_1^\infty I_{i_0}(\tau)d\tau\le { J}_i(1)+C_1\tilde{ J}_{i_0}, \quad\forall\; i\in\Omega.
$$
Hence, $|\tilde{ J}_i|<\infty$ for all $i\in\Omega$ and $\|\bm {\tilde J}\|_\infty<\infty$. 
Since $|I'_i|$ is uniformly bounded for $t\in (0, \infty)$ and $\tilde J_i<\infty$, it follow from Lemma \ref{lemma_inf0} that $I_i(t)\to 0$ as $t\to\infty$ for each $i\in\Omega$.

By \eqref{maxmi} again, we have
$$
\|\tilde{\bm J}-\bm J(t)\|_\infty\le  C_1\int_t^\infty  I_{i_0}(\tau)d\tau\to 0\ \ \text{as}\ \ t\to\infty. 
$$
By the first equation of \eqref{model-mass-action}, we have 
$$
\frac{d}{dt}( S_i- r_i)=- \beta_i( S_i- r_i) I_i, \quad i\in\Omega, t>0. 
$$
It follows that 
\begin{equation}\label{Sic}
 S_i- r_i=( S_{i}^0-r_i) e^{-\beta_i J_i(t)}, \quad i\in\Omega, t>0. 
\end{equation}
This implies that 
$$
 S_i\to  S_i^*:= r_i+( S_{i}^0- r_i) e^{- \beta_i \tilde{ J}_i}=e^{-\beta_i \tilde{ J}_i} S_{i}^0+(1-e^{-\beta_i \tilde{ J}_i})r_i \quad \text{as}\quad t\to\infty, \quad i\in\Omega. 
$$
Define $ \lambda_i^*=e^{- \beta_i \tilde{ J}_i}$, $i\in\Omega$. Then, $\bm 0\ll \bm \lambda\ll \bm 1$, $\bm S^*=\bm\lambda^*\circ \bm S^0+({\bm 1}-\bm \lambda^*)\circ \bm r$, and $\bm S(t)\to \bm S^*$ as $t\to\infty$. Noticing $\sum_{i\in\Omega}(S_i+I_i)=N$ for any $t\ge 0$ and $I_i(t)\to 0$ as $t\to\infty$, we have $\sum_i S_i^*=N$.  

Let $\bm\psi$ be the positive eigenvector corresponding with $s:=s(d_I\mathcal{L}^T+ {\rm diag}(\bm\beta\circ \bm\lambda^*\circ (\bm S^0-\bm r)))$ with $\bm \psi_M=1$. Note that $s=s(d_I\mathcal{L}+{\rm diag}(\bm\beta\circ\bm\lambda^*\circ(\bm S^0-\bm r)))$. Then, we have 
\begin{eqnarray*}
\frac{d}{dt}\sum_{i\in\Omega} \psi_iI_i&=& d_I\sum_{i, j\in\Omega}L_{ij}\psi_i I_j+\sum_{i\in\Omega}(\beta_i S_i-\gamma_i)\psi_i I_i\\
&=& d_I\sum_{i, j\in\Omega}L_{ji}\psi_j I_i+\sum_{i\in\Omega}(\beta_i S_i-\gamma_i)\psi_i I_i\\
&=&s\sum_{i\in\Omega}\psi_i I_i+\sum_{i\in\Omega}\beta_i( S_i- S_{i}^*)\psi_i I_i\ge\left(s-\bm\beta_M\|\bm S-\bm S^*\|_\infty\right)\sum_{i\in\Omega}\psi_i I_i.
\end{eqnarray*}
Therefore, 
$$
N\ge \sum_{i\in\Omega} \psi_i I_i\ge e^{\int_0^t( s-\bm\beta_M\|\bm S(\tau)-\bm S^*\|_\infty )d\tau} \sum_{i\in\Omega} \psi_i I_{i}^0, \quad t\ge 0. 
$$
Taking log for both sides and dividing by $t$, we obtain
$$
\frac{\ln N-\ln \sum_{i\in\Omega}\psi_i I_{i}^0}{t}+\beta_M\frac{\int_0^t \|\bm S(\tau)-\bm S^*\|_\infty d\tau}{t}\ge s.
$$
Letting $t\to\infty$ and noticing $\bm S(t)\to\bm S^*$ as $t\to\infty$, we have $s\le 0$.  

\noindent\textbf{Case 2.} $\tilde{J}_{i}=\infty$ for all $i\in\Omega$. Fix $i_1\in\Omega$. Then,  $J_{i_1}(t)\to\infty$ as $t\to\infty$. By \eqref{maxmi} and \eqref{Sic}, 
$$
\|\bm S(t)-\bm r\|_\infty\le \|\bm S_0-\bm r\|_\infty e^{-\frac{\bm\beta_m}{C_1} J_{i_1}(t)}\to 0\quad \text{as}\ t\to\infty. 
$$
Noticing $\sum_{i\in\Omega}(S_i+I_i)=N$ for all $t\ge 0$, $\sum_{i\in\Omega}I_i\to N-\sum_{i\in\Omega}r_i$ as $t\to\infty$. 
By Lemma \ref{lem0} and the second equation of \eqref{model-mass-action}, we have
    $$ 
    \lim_{t\to\infty}\Big\|\bm I(t)-{\Big(N-\sum_{i\in\Omega}r_i
    \Big)}\bm\alpha\Big\|_\infty=0.
    $$
This proves (ii). 

Finally, it is easy to see that $N\le \sum_{i\in\Omega}r_i$ implies (i). If (i) holds, then $N=\sum_{i\in\Omega} S^*_i=\sum_{i\in\Omega}(\lambda^*_iS_{i}^0+(1-\lambda^*_i)r_i)$ for some $\bm\lambda\in\mathbb{R}^n$ with $\bm 0\ll \bm\lambda^*\ll \bm 1$ and $s\le 0$. So, we have $N\le N^*_{\bm S^0,\bm r}$. Therefore, if $N> N^*_{\bm S^0,\bm r}$, then  (ii) holds. 
\end{proof}

\begin{rk} By Theorem \ref{thm-mass-ds}, whether limiting the movement of susceptible people can eliminate the disease is depending on the total population $N$: if  $N\le \sum_{j\in\Omega}r_{j}$, then disease can be controlled; if $N$ is large, then it cannot be controlled by limiting the movement of susceptible people. These predictions agree with those in \cite{li2023sis,Salako2023dynamics} where the predictions are based on the asymptotic profiles of the EE as $d_S\to 0$. 
\end{rk}

\begin{rk}
 Alternatively, we may use a Lyapunov function to prove Theorem \ref{thm-mass-ds}.  Define 
    $$
   V(\bm S, \bm I) =\sum_{i\in\Omega} \left(\frac{ \beta_i S_i^2}{2\gamma_i}+I_i\right).
    $$
   Then it is easy to check that 
    $$
    \dot{V}(\bm S, \bm I)=-\sum_{i\in\Omega}\frac{1}{\gamma_i}\left(\beta_i S_i-\gamma_i\right)^2 I_i\le 0.
    $$
\end{rk}

\section{Model with the standard incidence mechanism}
In the current section, we focus on the model with  standard incidence mechanism, that is 
 \begin{equation}\label{model-standard}
     \begin{cases}
         \bm S'=d_S\mathcal{L}\bm S-(\bm\beta\circ\bm S\circ\bm I)/(\bm S+\bm  I)+\bm\gamma\circ \bm I & t>0,\cr 
         \bm I'=d_I\mathcal{L}\bm I+(\bm\beta\circ\bm S\circ \bm I)/(\bm S+\bm I)-\bm\gamma\circ \bm I & t>0.
     \end{cases}
 \end{equation}
We will study the global dynamics of the model when $d_S=0$ or $d_I=0$ as in the previous section.  

 Here, we define the low-risk, moderate-risk, and high-risk patches ${H}^+$, $H^0$ and ${H}^-$  as 
$$
H^-:=\{i\in\Omega: 1<r_i\},\quad  H^0:=\{i\in\Omega  : 1=r_i \},\quad \text{and}\quad  H^{+}:=\{i\in\Omega : 1>r_i \}.
$$
The sets $H^{-}$, $H^{0}$, and $H^{+}$ partition $\Omega$ into disjoint subsets, and  either of these sets can be empty. Note that $1/r_i=\beta_i/\gamma_i$ is the basic reproduction number 
of patch $i$ when the system is decoupled.



\begin{tm}\label{thm-standard-di}
Suppose that {\bf (A1)-(A3)} holds, $d_I=0$ and $d_S>0$. Let $(\bm S, \bm I)$ be the solution of \eqref{model-standard}.  Then, 
$$
\lim_{t\to\infty}\bm S(t)=k\bm \alpha,\quad 
\lim_{t\to\infty} I_i(t)=0, \quad \forall\;i\in H_1, \quad \text{and}\quad \lim_{t\to\infty} I_i(t)=\frac{\beta_i-\gamma_i}{\gamma_i}k\alpha_i,\quad \forall\;i\in H_2,
$$
where  $H_1:=H^0\cup H^-\cup\Omega_{\bm I^0}^0$, $H_2:=H^+\cap\Omega^+_{\bm I^0}$  and 
$ k=N/({1+\sum_{i\in H_2} \frac{\beta_i-\gamma_i}{\gamma_i}\alpha_i}).$ 
\end{tm}
\begin{proof}
First, we claim that there exists $\underline S>0$ such that $\liminf_{t\to\infty} S_i(t)\ge \underline S$ for all $i\in\Omega$. To see this, by the equation of $\bm S$, 
\begin{align*}
\frac{d}{dt}\sum_{i\in\Omega} S_i(t) &= \sum_{i\in\Omega} \gamma_iI_i-\sum_{i\in\Omega}\beta_i\frac{I_i}{I_i+S_i}S_i\medskip 
\displaystyle \ge \bm\gamma_m\sum_{i\in\Omega} I_i-\bm\beta_M \sum_{i\in\Omega} S_i \medskip \\ 
\displaystyle &= \bm\gamma_m\left(N-\sum_{i\in\Omega} S_i\right)-\bm\beta_M \sum_{i\in\Omega} S_i \medskip= \bm\gamma_mN-(\bm\beta_M+\bm\gamma_m) \sum_{i\in\Omega} S_i.
\end{align*}
So we can choose $t_0>0$ such that 
\begin{equation}\label{t00}
 \sum_{i\in\Omega} S_i(t)\ge \frac{\bm\gamma_mN}{2(\bm\beta_M+\bm\gamma_m)}, \quad\forall\; t\ge t_0. 
\end{equation}
By Lemma \ref{Harnck-lemma}, there exists $c_0>0$ such that  
\begin{equation}\label{t01}
e^{d_S\mathcal{L}}\bm S(t)\ge c_0\|e^{d_S\mathcal{L}}\bm S(t)\|_1 \bm 1, \quad \forall\;t\ge 0. 
\end{equation}
By the equation $\bm S$ again,
$$
\bm S'\ge d_S\mathcal{L}\bm S-\bm\beta_M \bm S. 
$$
So by the comparison principle and \eqref{t00}-\eqref{t01}, we have
\begin{align*}
\bm S(t+1)&\ge e^{-\bm\beta_M } e^{d_S\mathcal{L}}\bm S(t)\ge c_0 e^{-\bm\beta_M } \|e^{d_S\mathcal{L}}\bm S(t)\|_1 \bm 1\\
&=c_0 e^{-\bm\beta_M } \|\bm S(t)\|_1 \bm 1\ge c_0 e^{-\bm\beta_M } \frac{\bm\gamma_mN}{2(\bm\beta_M+\bm\gamma_m)}\bm 1, \ \ \forall\; t\ge t_0. 
\end{align*}
It follows that 
\begin{equation}\label{sinf}
    S_i(t)\ge \underline S:= c_0 e^{-\bm\beta_M } \frac{\bm\gamma_mN}{2(\bm\beta_M+\bm\gamma_m)}, \quad\forall\; i\in\Omega, \ t\ge t_0+1. 
\end{equation}

    Define 
    $$
    V(\bm S, \bm I)=\frac{1}{2}\sum_{i\in\Omega}\theta_iS_i^2+\frac{1}{2}\sum_{i\in H_2}\theta_i
    \frac{\gamma_i}{\beta_i-\gamma_i}I_i^2,
    $$
    where $\bm\theta=1/\bm \alpha$.
Then, we have
\begin{equation}\label{Vp}
\dot V(\bm S, \bm I)=\sum_{i, j\in\Omega} \theta_i L_{ij}S_iS_j-\sum_{i\in H_1} \theta_i\frac{(\beta_i-\gamma_i)S_i-\gamma_i I_i}{S_i+I_i}S_iI_i-\sum_{i\in H_2}\frac{\theta_i}{\beta_i-\gamma_i}\frac{[(\beta_i-\gamma_i)S_i-\gamma_i I_i]^2I_i}{S_i+I_i}.
\end{equation}
By Theorem \ref{theorem_positive}, $\sum_{i, j\in\Omega} \theta_i L_{ij}S_iS_j\le 0$, where the equality holds if and only if $\bm S$ is a multiple of $\bm\alpha$.

It is easy to see that if $I_i(0)=0$ for some $i\in\Omega$ then $I_i(t)=0$ for all $t>0$. If $i\in H^-\cup H^0$, then $\beta_i\le\gamma_i$ and, by the equation of $\bm I$,
$$
I'_i(t)=\frac{((\beta_i-\gamma_i)S_i-\gamma_i I_i)I_i}{S_i+I_i}\le \frac{-\gamma_i I_i^2}{S_i+I_i}\le -\frac{\gamma_i}{N}I_i^2.
$$
It follows that $I_i(t)\to 0$ as $t\to\infty$ for all $i\in H_1$. Moreover, integrating the equation of $I_i$ on $(0, \infty)$, we find that 
$$
-I^0_i=\int_0^\infty I_i'(t) dt=\int_0^\infty \frac{(\beta_i-\gamma_i)S_i-\gamma_i I_i}{S_i+I_i}I_i dt, \quad \forall\;i\in H^-\cup H^0. 
$$
Noticing $(\beta_i-\gamma_i)S_i-\gamma_i I_i\le 0$ for $i\in H^-\cup H^0$, we have 
$$
0\le -\int_0^\infty \frac{(\beta_i-\gamma_i)S_i-\gamma_i I_i}{S_i+I_i}S_iI_i dt<\infty. 
$$
So integrating \eqref{Vp} over $(0, \infty)$, we find that 
\begin{equation}\label{lb}
0\le -\int_0^\infty \sum_{i, j\in\Omega} \theta_i L_{ij}S_iS_jdt<\infty\quad
\text{and}\quad  
\int_0^\infty \frac{[(\beta_i-\gamma_i)S_i-\gamma_i I_i]^2I_i}{S_i+I_i} dt<\infty, \quad\forall\; i\in H_2. 
\end{equation}
Since $\underline{S}/2\le S_i(t)+I_i(t)<N$ for all $t\gg 1$, we have  $\sup_{t\ge 0}|S_i'(t)|<\infty$ and $\sup_{t\ge 0}|I_i'(t)|<\infty$. We then deduce from \eqref{lb}  and Lemma \ref{lemma_inf0} that 
\begin{equation}\label{SS}
    \lim_{t\to\infty}\sum_{i, j\in\Omega} \theta_i L_{ij}S_iS_j=0 
\end{equation}
and 
\begin{equation}\label{II}
    \lim_{t\to\infty}\frac{[(\beta_i-\gamma_i)S_i-\gamma_i I_i]^2I_i}{S_i+I_i}=0, \quad\forall\; i\in H_2.
\end{equation}
Observe that 
$$
I_i'(t)=((\beta_i-\gamma_i)S_i-\gamma_iI_i)\frac{I_i}{S_i+I_i}\ge \Big(\frac{\underline{S}}{2}(\beta_i-\gamma_i)-\gamma_iI_i\Big)\frac{I_i}{S_i+I_i},\quad t\gg 1, \ i\in H_2.
$$
Therefore, $\liminf_{t\to\infty}I_i(t)\ge \frac{\underline{S}(\beta_i-\gamma_i)}{2\gamma_i}>0$ for $i\in H_2$. As a result, we obtain from \eqref{II} that 
\begin{equation}\label{II-2}
    \lim_{t\to\infty}[(\beta_i-\gamma_i)S_i-\gamma_i I_i]^2=0, \quad\forall\; i\in H_2.
\end{equation}
Since $\bm 0<\underline{S}\bm 1\le \liminf_{t\to\infty}\bm S(t)\le\limsup_{t\to\infty}\bm S(t)\le N\bm 1$, possibly after passing to a subsequence, we may suppose that $\bm S(t)\to \bm S^*$  as $t\to\infty$. Thanks to \eqref{SS},  we have $\sum_{i,j\Omega}\theta_iL_{ij}S_i^*S_j^*=0$. In view of Theorem \ref{theorem_positive},   $\bm S^*=k^*\bm \alpha$ for some positive number $k^*$.  By \eqref{II-2}, we have that $I_i(t)\to
\frac{k^*(\beta_i-\gamma_i)}{\gamma_i}$ as $t\to \infty$, for each $i\in H_2$.  Recalling that $\|\bm S(t)+\bm I(t)\|_1=N$ for all $t\ge 0$, then $N=k^*(1+\sum_{j\in H_2}(\beta_j-\gamma_j)\alpha_j/\gamma_j)$, which yields that $k^*=N/((1+\sum_{j\in H_2}(\beta_j-\gamma_j)\alpha_j/\gamma_j))$. Since $k^*$ is independent of the chosen subsequence, the desired result follows.

\end{proof}

\begin{rk} By Theorem \ref{thm-standard-di}, the  infected people will live exactly on $H^+\cap \Omega^+_{\bm I^0}$ (i.e. the high-risk patches with  positive initial infected cases) for the model with standard incidence mechanism when the dispersal of infected people is limited.
 As before, these predictions  agree with previous studies for the model with standard incidence mechanism in \cite{li2019dynamics,chen2020asymptotic}, which are based on the asymptotic profiles of the EE as $d_I\to 0$. The main difference is that our results are depending on the distribution of initial data $\bm I^0$.
\end{rk}


Next, we examine the scenario of a total restriction of the movement of the susceptible population. In this direction, our result reads as follows.





\begin{tm}\label{thm-standard-ds} Suppose that {\bf (A1)-(A3)} holds, $d_S=0$ and $d_I>0$. Let $(\bm S(t),\bm I(t))$ be the solution of \eqref{model-eq1}. Then the following conclusions hold:
\begin{enumerate}
\item[\rm (i)] If $H^-\cup H^0$ is nonempty, then $\bm I(t)\to\bm 0$ as $t\to\infty$ and $\liminf_{t\to\infty} S_j(t)>0$  for $j\in H^-\cup H^0$.


\item[\rm (ii)]  If $H^+=\Omega$,  then  $(\bm S(t),\bm I(t))\to ( I^*(\bm r\circ\bm\alpha/(\bm 1-\bm r)),I^*\bm \alpha)$  as $t\to\infty$, where $I^*=N/(1+\|\bm 1/(\bm 1-\bm r)\|_1)$.
\end{enumerate}
    
\end{tm}
\begin{proof}
    Define 
    $$
    V(\bm S,\bm I)=\frac{1}{2}\sum_{i\in\Omega}\theta_i\frac{(1-r_i)}{r_i}S_i^2+\frac{1}{2}\sum_{i\in\Omega}\theta_i I_i^2,
    $$
    where $\bm \theta=\bm 1/\bm\alpha$.
    Then $-\frac{N^2}{2}\sum_{i\in\Omega}\theta_i\le V(\bm S,\bm I)\le N^2(1+\|\bm 1/\bm r\|_{\infty})\sum_{i\in\Omega}\theta_i $ and
    \begin{equation*}
        \dot{V}(\bm S,\bm I)=\sum_{i,j\in\Omega}\theta_iL_{ij}I_iI_j-\sum_{i\in\Omega}\frac{\theta_i\beta_i^2}{\gamma_i}\big((1-r_i)S_i-r_i I_i\big)^2\frac{I_i}{S_i+I_i}.
    \end{equation*}
    By Theorem \ref{theorem_positive}, $\sum_{i, j\in\Omega} \theta_i L_{ij}I_iI_j\le 0$, where the equality holds if and only if $\bm I$ is a multiple of $\bm\alpha$.   Denote $\mathcal{N}:=\{(\bm S, \bm I)\in \mathcal{E}:\ \dot{V}(\bm S, \bm I)=0\}$. Then, 
$$
\mathcal{N}=\left\{(\bm S, \bm I)\in \mathcal{E}:\ \big((1-r_i)S_i-r_i I_i\big)^2I_i=0 \ \ \forall i\in\Omega\ \ \text{and}\ \bm I=a\bm\alpha \text{ for some } a\in\mathbb{R}_+ \right\}.
$$
From this point, we distinguish two cases.

{\rm (i)} Suppose that $H^-\cup H^{0}\ne\emptyset$. In this case, the maximal connected invariant subset of $\mathcal{N}$ is 
$$
\mathcal{N}_1:=\left\{(\bm S, \bm I)\in \mathcal{E}:\  \bm I=\bm 0 \right\}.
$$
Hence by the invariance principle \cite[Theorem 4.3.4]{DanHenry}, $(\bm S(t), \bm I(t))\to \mathcal{N}_1$ as $t\to\infty$. Thus $\bm I(t)\to \bm 0$ as $t\to\infty$. 

Next, fix $i\in H^{-}\cup H^{0}$. Then
$$ 
S_i'=(\gamma_iI_i-(\beta_i-\gamma_i)S_i)\frac{I_i}{S_i+I_i}\ge\frac{\gamma_i I_i^2}{S_i+I_i} \ge \frac{\bm \gamma_m }{N}I_i^2.
$$
Integrating both sides and taking $\liminf$ as $t\to\infty$, we obtain
$$
\liminf_{t\to\infty}S_i(t)\ge S_{i}^0+\frac{\bm \gamma_m}{N}\int_{0}^{\infty}I_i^2dt>0.
$$

{\rm (ii)} Suppose that $H^-\cup H^{0}=\emptyset$, that is $H^+=\Omega$. Then, in addition to the set $\mathcal{N}_1$, the set $\mathcal{N}_2:=\{( I^*(\bm r\circ\bm\alpha/(\bm 1-\bm r)),I^*\bm \alpha)\}$
is also a maximal connected invariant subset of $ \mathcal{N}$, where $I^*=N/(1+\|\bm 1/(\bm 1-\bm r)\|_1)$. Hence by the invariance principle \cite[Theorem 4.3.4]{DanHenry}, $(\bm S(t), \bm I(t))\to \mathcal{N}_i$ for some $i=1,2$ as $t\to\infty$. Observe that $\mathcal{N}_2$ contains a single point.  To complete the proof of the result, it remains to show that $\bm I(t) \nrightarrow \bm 0$ as $t\to\infty$. To this end, setting $F_i(t)=I_i/(S_i+I_i)$ for $t>0$ and $i\in\Omega$, then $(\bm S(t), \bm I(t))$ solves the cooperative system
\begin{equation}\label{coop-sys}
    \begin{cases}
        \bm S'=(\bm\gamma\circ \bm I-(\bm \beta-\bm \gamma)\circ\bm S)\circ \bm F, & t>0,\cr 
        \bm I'=d_I\mathcal{L}\bm I +((\bm \beta-\bm \gamma)\circ\bm S-\bm\gamma\circ \bm I)\circ \bm F, & t>0.
    \end{cases}
\end{equation}
Taking 
$$
(\underline{\bm S},\underline{\bm I})=(l(\bm\gamma/(\bm\beta-\bm\gamma))\circ\bm\alpha,l\bm \alpha)
$$
where $l:=\min\{(\bm S(1)\circ(\bm\beta-\bm\gamma)/\bm \gamma\circ \bm\alpha)_m, (\bm I(1)/\bm \alpha)_m\}>0$, we have that $(\underline{\bm S},\underline{\bm I}) $ solves \eqref{coop-sys} and $(\underline{\bm S},\underline{\bm I})\le (\bm S(1),\bm I(1))$.  Therefore, by the comparison principle for cooperative systems, we have that $ (\underline{\bm S},\underline{\bm I})\le (\bm S(t),\bm I(t))$ for all $t\ge 1$. This shows that the omega limit set of $(\bm S(t),\bm I(t))$ doesn't intersect with $\mathcal{N}_1$. Therefore, $(\bm S(t),\bm I(t))\to \mathcal{N}_2$ as $t\to\infty$.

\end{proof}
\begin{rk}By Theorem \ref{thm-standard-ds}, if $H^-\cup H^0$ is nonempty, the disease can be controlled by limiting the movement of susceptible people.  This  agrees with previous studies for the model with standard incidence mechanism in \cite{allen2007asymptotic,chen2020asymptotic}, which are based on the asymptotic profiles of the EE as $d_I\to 0$. If $H^-\cup H^0$ is empty, then the disease cannot be controlled by limiting the movement of susceptible people. 
    
\end{rk}

\bibliographystyle{plain}
 \bibliography{epidem}

\begin{thebibliography}{10}

\bibitem{allen2007asymptotic}
L.~J.~S. Allen, B.~M. Bolker, Y.~Lou, and A.~L. Nevai.
\newblock Asymptotic profiles of the steady states for an {S}{I}{S} epidemic
  patch model.
\newblock {\em SIAM Journal on Applied Mathematics}, 67(5):1283--1309, 2007.

\bibitem{Allen}
L.~J.~S. Allen, B.~M. Bolker, Y.~Lou, and A.~L. Nevai.
\newblock Asymptotic profiles of the steady states for an {SIS} epidemic
  reaction-diffusion model.
\newblock {\em Discrete Contin. Dyn. Syst.}, 21(1):1--20, 2008.

\bibitem{castellano2022effect}
K.~Castellano and R.~B. Salako.
\newblock On the effect of lowering population's movement to control the spread
  of an infectious disease.
\newblock {\em Journal of Differential Equations}, 316:1--27, 2022.

\bibitem{chen2020asymptotic}
S.~Chen, J.~Shi, Z.~Shuai, and Y.~Wu.
\newblock Asymptotic profiles of the steady states for an {S}{I}{S} epidemic
  patch model with asymmetric connectivity matrix.
\newblock {\em Journal of Mathematical Biology}, 80(7):2327--2361, 2020.

\bibitem{DBS2023}
J.~T. Doumat\`e, T.~B. Issa, and R.~B. Salako.
\newblock Competition-exclusion and coexistence in a two-strain {S}{I}{S}
  epidemic model in patchy environments.
\newblock {\em arXiv:2308.10348}.

\bibitem{gao2020does}
Daozhou Gao.
\newblock How does dispersal affect the infection size?
\newblock {\em SIAM Journal on Applied Mathematics}, 80(5):2144--2169, 2020.

\bibitem{gao2021impact}
Daozhou Gao and Yuan Lou.
\newblock Impact of state-dependent dispersal on disease prevalence.
\newblock {\em Journal of Nonlinear Science}, 31(5):73, 2021.

\bibitem{guo2006global}
Hongbin Guo, Michael~Y Li, and Zhisheng Shuai.
\newblock Global stability of the endemic equilibrium of multigroup sir
  epidemic models.
\newblock {\em Canadian applied mathematics quarterly}, 14(3):259--284, 2006.

\bibitem{guo2012global}
Hongbin Guo, Michael~Y Li, and Zhisheng Shuai.
\newblock Global dynamics of a general class of multistage models for
  infectious diseases.
\newblock {\em SIAM Journal on applied mathematics}, 72(1):261--279, 2012.

\bibitem{DanHenry}
D.~Henry.
\newblock Geometric theory of semilinear parabolic equations.
\newblock 1981.

\bibitem{li2019dynamics}
H.~Li and R.~Peng.
\newblock Dynamics and asymptotic profiles of endemic equilibrium for {S}{I}{S}
  epidemic patch models.
\newblock {\em Journal of mathematical biology}, 79:1279--1317, 2019.

\bibitem{li2023sis}
H.~Li and R.~Peng.
\newblock An {S}{I}{S} epidemic model with mass action infection mechanism in a
  patchy environment.
\newblock {\em Studies in Applied Mathematics}, 150(3):650--704, 2023.

\bibitem{Li2018}
H.~Li, R.~Peng, and Z.~Wang.
\newblock On a diffusive susceptible-infected-susceptible epidemic model with
  mass action mechanism and birth-death effect: analysis, simulations, and
  comparison with other mechanisms.
\newblock {\em SIAM J. Appl. Math.}, 78(4):2129--2153, 2018.

\bibitem{li2010global}
M.~Y. Li and Z.~Shuai.
\newblock Global-stability problem for coupled systems of differential
  equations on networks.
\newblock {\em J. Differential Equations}, 248(1):1--20, 2010.

\bibitem{Salako2023degenerate}
R.~Salako and Y.~Wu.
\newblock On degenerate reaction-diffusion epidemic models with mass action or
  standard incidence mechanism.
\newblock Submitted.

\bibitem{Salako2023dynamics}
R.~Salako and Y.~Wu.
\newblock On the dynamics of an epidemic patch model with mass-action
  transmission mechanism and asymmetric dispersal patterns.
\newblock Submitted.

\bibitem{shu2012global}
Hongying Shu, Dejun Fan, and Junjie Wei.
\newblock Global stability of multi-group seir epidemic models with distributed
  delays and nonlinear transmission.
\newblock {\em Nonlinear Analysis: Real World Applications}, 13(4):1581--1592,
  2012.

\bibitem{shuai2013global}
Zhisheng Shuai and Pauline van~den Driessche.
\newblock Global stability of infectious disease models using lyapunov
  functions.
\newblock {\em SIAM Journal on Applied Mathematics}, 73(4):1513--1532, 2013.

\bibitem{WuZou}
Y.~Wu and X.~Zou.
\newblock Asymptotic profiles of steady states for a diffusive {SIS} epidemic
  model with mass action infection mechanism.
\newblock {\em J. Differential Equations}, 261(8):4424--4447, 2016.

\bibitem{zhang2015graph}
Chunmei Zhang, Wenxue Li, and Ke~Wang.
\newblock Graph-theoretic method on exponential synchronization of stochastic
  coupled networks with markovian switching.
\newblock {\em Nonlinear Analysis: Hybrid Systems}, 15:37--51, 2015.

\end{thebibliography}

\end{document}